\documentclass[a4paper,11pt]{article}
\usepackage{amsmath,amsthm,amssymb}
\usepackage{hyperref}

\parskip \bigskipamount

\makeatletter
\def\thm@space@setup{%
  \thm@preskip=\parskip \thm@postskip=0pt
}
\makeatother

\newtheorem{thm}{Theorem}

\newtheorem{corollary}[thm]{Corollary}

\theoremstyle{remark}  

\title{Odd Covers of Complete Graphs and Hypergraphs}
\author{Imre Leader\thanks{Department of Pure Mathematics and Mathematical Statistics, Centre for Mathematical Sciences, University of Cambridge, Wilberforce Road, Cambridge CB3 0WB, United Kingdom. Email: \texttt{I.Leader@dpmms.cam.ac.uk}.} \and Ta Sheng Tan\thanks{Institute of Mathematical Sciences, Faculty of Science, Universiti Malaya, 50603 Kuala Lumpur, Malaysia. Email: \texttt{tstan@um.edu.my}.}}

\begin{document}

\maketitle

\begin{abstract}
 The `odd cover number' of a complete graph is the smallest size of a family of complete bipartite graphs that covers each edge an odd number of times.
 For $n$ odd, Buchanan, Clifton, Culver, Nie, O'Neill, Rombach and Yin showed that the odd cover number of $K_n$ is equal to $(n+1)/2$ or $(n+3)/2$, and they conjectured that it is always $(n+1)/2$. We prove this conjecture.
 
 For $n$ even, Babai and Frankl showed that the odd cover number of $K_n$ is always at least $n/2$, and the above authors and Radhakrishnan, Sen and Vishwanathan gave some values of $n$ for which equality holds. We give some new examples.

 Our constructions arise from some very symmetric constructions for the corresponding problem for complete hypergraphs.
 Thus the odd cover number of the complete 3-graph $K_n^{(3)}$ is the smallest number of complete 3-partite 3-graphs such that each 3-set is in an odd number of them.
 We show that the odd cover number of $K_n^{(3)}$ is exactly $n/2$ for even $n$, and we show that for odd $n$ it is $(n-1)/2$ for infinitely many values of $n$.
 We also show that for $r=3$ and $r=4$ the odd cover number of $K_n^{(r)}$ is strictly less than the partition number, answering a question of Buchanan, Clifton, Culver, Nie, O'Neill, Rombach and Yin for those values of $r$.
\end{abstract}


\section{Introduction}
 The \emph{odd cover number} $b(n)$ of the complete graph $K_n$ is the minimum number of complete bipartite graphs needed so that each edge is in an odd number of these complete bipartite graphs.
 By considering the ranks of the adjacency matrices, Babai and Frankl~\cite{babai} showed that $b(n)$ is always at least $\lfloor n/2\rfloor$. 
 
 For $n$ even, 
 Radhakrishnan, Sen and Vishwanathan~\cite{radhakrishnan} gave infinitely many values of $n$ for which equality holds.
 More specifically, they showed that $b(n) = n/2$ whenever $n = 2(q^2+q+1)$ for some prime power $q\equiv 3\pmod 4$, and also whenever there exists a Hadamard matrix of order $n/2$. Note that the
 former values are all $2 \pmod 8$ while the latter are $0 \pmod 8$. These results were extended
 by Buchanan, Clifton, Culver, Nie, O'Neill, Rombach and Yin~\cite{buchanan}, who gave an
 elegant construction to show that
 in fact we always have $b(n)=n/2$ when $n \equiv 0 \pmod 8$. They also showed that for every
 even value of $n$ we have $b(n)=n/2$ or $(n/2)+1$.

 For $n$ odd, the authors of \cite{buchanan} showed that $b(n)$ is always equal to $(n+1)/2$ or
 $(n+3)/2$. They also showed that $b(n)=(n+1)/2$ whenever $n \equiv \pm 1 \pmod 8$. (These are
 immediate from their result about $n \equiv 0 \pmod 8$, since clearly $b(n-1) \le b(n)$ by removing
 a vertex, and also $b(n+1) \le b(n)+1$ by adding a new vertex and a star at that vertex.)
 They made the conjecture that in fact $b(n)=(n+1)/2$ for all odd values of $n$. Note that this
 cannot follow in a similar way by adding or removing a vertex from even cases, because there are
 several even values of $n$ for which we do not have $b(n)=n/2$ -- for example, this conjecture would
 assert that $b(13)=7$, whereas $b(12)=7$ and $b(14)=8$.

 Generalising the above odd cover problem to hypergraphs, fix $r \ge 2$: we now wish to cover the
 collection $K_n^{(r)}$ of all the $r$-sets from an $n$-set. 
 We say a set of complete $r$-partite $r$-graphs is an \emph{odd cover} of $K_n^{(r)}$ if each $r$-set is in an odd number of these complete $r$-partite $r$-graphs. The \emph{odd cover number} $b_r(n)$ is the smallest size of an odd cover of $K_n^{(r)}$.
 (Here as usual a \emph{complete $r$-partite $r$-graph} is specified by $r$ disjoint sets $A_1,A_2, \ldots A_r$ of vertices, and consists of those $r$-sets that meet every $A_i$.)

 By taking the `link' of a vertex, it is easy to see that $b_{r}(n)\ge b_{r-1}(n-1)$.
 Indeed, given an odd cover of $K_{n}^{(r)}$, let $v$ be a vertex, and for each complete $r$-partite $r$-graph in the cover, we form a complete $(r-1)$-partite $(r-1)$-graph by removing the class containing $v$.
 These complete $(r-1)$-partite $(r-1)$-graphs form an odd cover of $K_{n-1}^{(r-1)}$, as required.
 
 So in particular, $b_3(n)\ge \lfloor n/2\rfloor$ for all $n\ge 3$.
 We will show that this is sharp for all even $n$: $b_3(n) = n/2$.

 By taking a link, this implies that there is an odd cover of $K_{n-1}$ of size $n/2$ for 
 every even $n$, thus proving the above conjecture. What is interesting is that our constructions
 for $r=3$ are very symmetric, and this symmetry is somewhat lost when we pass to the link.
 So in a sense the key to understanding the graph case is to consider the 
 hypergraph case.
  
 We are also able to show that the above lower bound of $b_3(n) \ge \lfloor n/2 \rfloor$ is attained for infinitely many values of odd $n$.
 Specifically, we construct an odd cover of $K_n^{(3)}$ of size $(n-1)/2$ whenever 
 $n \equiv 1 \pmod 8$ and also when $n$ is of the form $3^k$. While the former comes from an
 analysis of odd covers of $K_n$ for $n \equiv 0 \pmod 8$, the latter is a direct construction,
 and it yields (by taking a link) new values in the graph case: it gives that $b(n)=n/2$ whenever
$n=3^k-1$, which are new values when $k$ is odd.
 
 The odd cover problem is related to the Graham-Pollak problem for partitioning the edge set of a complete hypergraph. 
 Let $f_r(n)$ be the minimum number of complete $r$-partite $r$-graphs needed to \emph{partition} the edge set of $K_n^{(r)}$.
 The Graham-Pollak theorem~\cite{graham1, graham2} asserts that $f_2(n) = n-1$. 
 For $r\ge 3$, Alon~\cite{alon1} showed that $f_3(n) = n-2$, and that
 $$\frac{2}{\binom{2\lfloor r/2 \rfloor}{\lfloor r/2\rfloor}}(1+o(1))\binom{n}{\lfloor r/2\rfloor}\le f_r(n)\le (1-o(1))\binom{n}{\lfloor r/2\rfloor}.$$
 There have been several improvements since Alon first studied the Graham-Pollak problem for hypergraphs. 
 See, for example, improvements to the lower bound in Cioab\v{a}, K\"{u}ndgen and Verstra\"{e}te~\cite{cioaba1} and improvements to the upper bound in Leader, Mili\'{c}evi\'{c} and Tan \cite{leader}, Leader and
 Tan \cite{leader2} and Babu and Vishwanathan \cite{babu}.
 Indeed, it is shown in \cite{leader2} that 
 $$f_r(n) \le c_r (1+o(1)) \binom{n}{\lfloor r/2 \rfloor} $$
 where the constants $c_r$ satisfy $c_r \le \frac{r}{2}(14/15)^{r/4}+o(1)$, so that in particular $c_r \rightarrow 0$
 as $r \rightarrow \infty$.
 
 Trivially, we must have $b_r(n)\le f_r(n)$.
 Having shown that $b(n) < f_2(n)$ for $n\ge 5$, Buchanan, Clifton, Culver, Nie, O'Neill, Rombach and Yin~\cite{buchanan} asked if this is also the case for larger uniformities.
 
 Our constructions above show that this is the case for $r=3$, giving $ b_3(n)\le \lceil n/2\rceil < f_3(n)$ for $n\ge 6$.
 We will also show that $b_4(n) < f_4(n)$ for $n$ sufficiently large.
 For values of $r$ greater than 4, the methods used in  \cite{leader2} for analysing $f_r(n)$ can be used to improve the upper bound on $b_r(n)$ mentioned in \cite{buchanan}. But unfortunately these improvements are not sufficiently strong to show that $b_r(n) < f_r(n)$ for any $r\ge 5$, because the
 known lower bounds on $f_r(n)$ are still too weak.

\section{Odd covers of complete 3-graphs}\label{sec2}

 In this section, we prove that $b_3(n) = \lfloor n/2\rfloor$ whenever $n$ is even, and also for $n = 3^k$ for any $k$ or $n\equiv 1\pmod{8}$.

 \begin{thm}\label{theorem_b3_even}
  For any even $n\ge 4$, there is an odd cover of $K_n^{(3)}$ of size $n/2$.
 \end{thm}
 \begin{proof}
  Let $n=2k$ and identify the vertices of $K_{2k}^{(3)}$ with $2k$ points around a circle.
  Consider the $k$ complete $3$-partite $3$-graphs, each obtained by taking the two points on a diameter as one part, while the two remaining intervals are the other two parts.
  More precisely, let the vertex set be $\mathbb{Z}_{2k}$, the set of integers modulo $2k$.
  For each $0\le i\le k-1$, let $H_i$ be the complete $3$-partite $3$-graph with parts $\{i, i+k\}$, $\{i+1, i+2\ldots, i+k-1\}$, and $\{i+k+1, i+k+2, \ldots, i-1\}$.
  We claim that $H_0, H_1, \ldots, H_{k-1}$ form an odd cover of $K_{2k}^{(3)}$.
  For convenience, we also write $H_{k+i}$ for $H_i$.

  Let $A$ be a $3$-set.
  If two points of $A$ are opposite (i.e., of the form $i$ and $k+i$), then it is clear that $A$ is in exactly one of the $H_i$, namely $H_{a}$, where $a$ is the other point of $A$.
  Now suppose no two points of $A$ are opposite.
  By rotation invariance, we may assume that $A = \{0,b,c\}$ for some $b < k$.
  It is then easy to check that if $k < c < k+b$ then $A$ is in three of the $H_i$, and otherwise $A$ is in exactly one of the $H_i$.
 \end{proof}

 As a consequence of Theorem~\ref{theorem_b3_even}, there is an odd cover of $K_{2k-1}$ of size $k$, by taking the link of any vertex in the odd cover of $K_{2k}^{(3)}$. This establishes Conjecture 7.2 in \cite{buchanan}.

 \begin{corollary}
  For any odd $n\ge 3$, we have $b(n) = (n+1)/2$. 
 \end{corollary}
 
 While our simple construction above for complete $3$-graphs is highly symmetric, the resulting odd covers for complete graphs are not.
 We do not see how to construct symmetric odd covers of $K_n$ of size $b(n)$ for odd $n$ in general.

 We now turn to $b_3(n)$ for odd values of $n$. 

 \begin{thm}\label{theorem_b3_power-three}
  Let $n = 3^k$ for some $k\ge 1$.
  Then there is an odd cover of $K_n^{(3)}$ of size $(n-1)/2$.
 \end{thm}
 \begin{proof}
  We identify the vertex set of $K_n^{(3)}$ with $\mathbb{F}_3^k$, the set of all vectors of length $k$ over the finite field $\mathbb{F}_3$.
  For each nonzero vector $x$, let $H_{x}$ be the complete $3$-partite $3$-graphs obtained by partitioning $\mathbb{F}_3^k$ into the three affine planes 
  $\{y\colon x \cdot y = 0\}$, $\{y\colon x \cdot y = 1\}$, and $\{y\colon x \cdot y = 2\}$.
  Here the dot product of two vectors $x$ and $y$ is defined by $x \cdot y = \sum x_iy_i$.
  Clearly, $H_{x} = H_{2x}$ for any $x\neq 0$, and so we have defined $(n-1)/2$ complete $3$-partite $3$-graphs.
  We claim that these $(n-1)/2$ complete $3$-partite $3$-graphs form an odd cover of $K_n^{(3)}$.

  Let $A$ be a $3$-set.
  Note that by affine invariance, we may assume that $A = \{0, a, b\}$.
  There are now two cases to check: either $b = 2a$, or else $a$ and $b$ are linearly independent.
  If $b = 2a$, then $A$ is in $H_{x} = H_{2x}$ if and only if $x \cdot a = 1$, and there are exactly $3^{k-1}$ such vectors $x$.
  For the case when $a$ and $b$ are linearly independent, $A$ is in $H_{x}$ if and only if $x \cdot a \neq0$, $x \cdot b \neq 0$, and $x \cdot a \neq x \cdot b$ -- there are exactly $3^{k-2}$ such vectors that are pairwise linearly independent.
 \end{proof}

 Again, by taking the link of a vertex, we deduce that there is an odd cover of $K_{3^k-1}$ of size $(3^k-1)/2$ for any $k\ge 1$.
 Observe that $3^k - 1\equiv 2\pmod{8}$ for odd $k$, and so we obtain infinitely many new values of $n\equiv 2\pmod{8}$ for which $b(n) = n/2$.

 \begin{corollary}
  Let $k\ge 1$ and $n = 3^k - 1$. Then $b(n) = n/2$.
 \end{corollary}

 For the case where $n\equiv 1\pmod{8}$, we first remark that our construction in Theorem~\ref{theorem_b3_even} is far from unique.
 Indeed, given any $k$ vectors $\{a^{(i)}\colon 1\le i\le k\}\subset \{0,1,-1\}^{2k}$ with the property that $a_j^{(i)} = 0$ if and only if $i=j$, and $a_j^{(i)} = -a_i^{(j)}$ for any $i\neq j$, we may identify the vertex set of $K_{2k}^{(3)}$ with $\{a^{(i)}\colon 1\le i\le k\}\cup\{-a^{(i)}\colon 1\le i\le k\}$.
 Then it is straightforward to check that the $2k$ natural tripartitions $A_j\cup B_j\cup C_j$ of the vertex set, where $A_j, B_j, C_j$ each consists of the vertices whose $j$th entry is $1, -1, 0$, respectively, induce an odd cover of $K_{2k}^{(3)}$. 
 For $n = 8k$, one such set of $4k$ vectors was used by Buchanan, Clifton, Culver, Nie, O'Neill, Rombach and Yin~\cite{buchanan} to show that $b(8k) = 4k$.
 More precisely, they defined the vectors $\{a^{(i)}\colon 1\le i\le 2k\}\subset \{0,1,-1\}^{4k}$ satisfying the above property to be such that for any $j>i$, $a_j^{(i)} = - 1$ if and only if $j\ge i+2$ or $j=i+1$ when $i\equiv 0,1\pmod{4}$.
 And then they showed that the $4k$ complete bipartite graphs with parts $A_j$ and $B_j$ form an odd cover of $K_{8k}$.
 This can then be easily extended to an odd cover of $K_{8k+1}^{(3)}$ by adding a new vertex $v$. 
 Indeed, one can check that the $4k$ complete $3$-partite $3$-graphs with parts $A_j$, $B_j$, and $C_j\cup \{v\}$ form an odd cover of $K_{8k+1}^{(3)}$.

 \begin{thm}\label{theorem-b3-1mod8}
  Let $n = 8k + 1$ for some $k\ge 1$. Then there is an odd cover of $K_n^{(3)}$ of size $4k$.\qed
 \end{thm}

 Since $\lfloor n/2\rfloor \le b_3(n) \le b_3(n+1)$ for any $n\ge 3$, the above results on $b_3(n)$ may be summarised as follows.

 \begin{corollary}
  For any even $n$ we have $b_3(n)=n/2$, while for any odd $n$ we have 
  $b_3(n)=(n-1)/2$ or $(n+1)/2$. Moreover, for odd $n$ that are either $1 \pmod 8$ or are a power of
  3 we have $b_3(n)=(n-1)/2$.
 \end{corollary}

 It would be very interesting to determine what happens for other values of $n$. 
 
\section{Odd covers of complete 4-graphs}

 In this section, we will show that $b_4(n) < f_4(n)$ for $n$ sufficiently large.
 The proof uses a similar approach to that used in \cite{leader} in proving an upper bound on $f_4(n)$.
 Let $h(m,n)$ be the minimum number of products of complete bipartite graphs (that is, sets of the form $E(K_{a,b})\times E(K_{c,d})$) needed so that each element of $E(K_m)\times E(K_n)$ is in an odd number of these products of complete graphs.
 Clearly, we have $h(m,n)\le b(m)b(n)\le \frac{(m+2)(n+2)}{4}$, by taking the product of complete bipartite graphs in the odd covers of $K_m$ and $K_n$.
 It turns out that this trivial upper bound for $h(m,n)$ is
 enough for our purposes.
 
 The proof of the following result is similar to
 the iteration used in Proposition 1 in \cite{leader}. 

 \begin{thm}\label{theorem_b4}
  $b_4(n) \le \frac{1}{4}(1+o(1))\binom{n}{2}$.
 \end{thm}
 \begin{proof}
  We will show that 
  \begin{equation}\label{eqn_b4}
   b_4(n)\le \frac{n^2}{8} + Cn\log n
  \end{equation}
  for some sufficiently large $C$.
  This is clearly true for $n\le 4$. So assume $n>4$ and the inequality \eqref{eqn_b4} holds for $1,2,\ldots,n-1$. 
  We will consider the case when $n$ is even - the case when $n$ is odd is similar.
  We say that a set of $4$-sets is \emph{odd covered} by a set of complete $4$-partite $4$-graphs if each of these $4$-sets is in an odd number of these complete $4$-partite $4$-graphs, and every other $4$-sets is in an even number of them.
  
  In order to odd cover the edge set of $K_n^{(4)}$, we can split the $n$ vertices into two equal parts, say $V\left(K_n^{(4)}\right) = A\cup B$, where $|A|=|B|= n/2$. The sets of 4-edges $\{e: e\subset A$\} and $\{e:e\subset B$\} can each be odd covered by $b_4(n/2)$ complete $4$-partite $4$-graphs; the sets of 4-edges $\{e: |e\cap A|=3\}$ and $\{e: |e\cap B|=3\}$ can each be odd covered by $b_3(n/2)$ complete 4-partite 4-graphs; while the remaining set of 4-edges $\{e: |e\cap A| = |e\cap B|=2\}$ can be odd covered by $h(n/2,n/2)$ complete $4$-partite $4$-graphs. 
  So by the induction hypothesis, and recalling the bounds $h(n/2,n/2)\le \frac{1}{4}(\frac{n}{2}+2)^2$ and $b_3(n/2)\le \frac{1}{2}(\frac{n}{2}+1)$, we have
  \begin{align*}
   b_4(n) &\le 2b_4(n/2) + h(n/2, n/2) + 2b_3(n/2)\\
	  &\le 2\left(\frac{n^2}{32}+\frac{Cn}{2}\log\left(\frac{n}{2}\right)\right) + \frac{1}{4}\left(\frac{n}{2}+2\right)^2 + \left(\frac{n}{2}+1\right)\\
	  &\le \frac{n^2}{8} + Cn\log n,
  \end{align*}
  completing the proof.
 \end{proof}

 Recalling that $f_3(n) = n-2$ and $f_4(n) \ge \frac{1}{3}(1+o(1))\binom{n}{2}$, our bounds for $b_3(n)$ in Section~\ref{sec2} and Theorem~\ref{theorem_b4} show that $b_r(n)$ is asymptotically smaller than $f_r(n)$ for $r = 3, 4$. This answers a question asked in \cite{buchanan}
 for these values of $r$.

Could this also be the case for larger values of $r$? 
 As mentioned in the Introduction, the methods used for proving asymptotic upper bounds for $f_r(n)$ can similarly be used to establish upper bounds for $b_r(n)$.
 In fact, observing that 
 $$b_{2s}(n) \le 2b_{2s}(n/2) + 2b_{2s-1}(n/2) + \sum_{t = 2}^{2s-2} b_t(n/2)b_{2s-t}(n/2),$$
 one can obtain that $b_{r}(n) \le \frac{1}{2^{r/2}}(1+o(1))\binom{n}{r/2}$ when $r$ is even.
 Unfortunately, this bound is not sufficiently strong enough to show that $b_r(n) < f_r(n)$ for general $r$.

\end{document}